\documentclass{amsart}
\usepackage{amssymb, colonequals, enumerate, mathptmx,stmaryrd}
\usepackage[nobysame]{amsrefs}

\usepackage[all]{xy}
\SelectTips{cm}{}

\usepackage[colorlinks]{hyperref}

\newtheorem*{Theorem}{Theorem}
\newtheorem{theorem}[subsection]{Theorem}
\newtheorem{proposition}[subsection]{Proposition}
\newtheorem{lemma}[subsection]{Lemma}
\newtheorem{corollary}[subsection]{Corollary}

\theoremstyle{definition}
\newtheorem{example}[subsection]{Example}
\newtheorem{definition}[subsection]{Definition}

\theoremstyle{remark}
\newtheorem{remark}[subsection]{Remark}
\newtheorem{claim}{Claim}

\numberwithin{equation}{subsection}

\newcommand{\lra}{\longrightarrow}
\newcommand{\xra}{\xrightarrow}

\newcommand{\bbZ}{\mathbb{Z}}
\newcommand{\fa}{\mathfrak{a}}
\newcommand{\fm}{\mathfrak{m}}
\newcommand{\fn}{\mathfrak{n}}
\newcommand{\fp}{\mathfrak{p}}

\newcommand{\fd}{\operatorname{flat\,dim}}
\newcommand{\hh}[1]{\operatorname{H}(#1)}
\newcommand{\HH}[2]{\operatorname{H}_{#1}(#2)}
\newcommand{\Hom}{\operatorname{Hom}}
\newcommand{\lch}[3]{\operatorname{H}_{#1}^{#2}(#3)}
\newcommand{\lotimes}{\otimes^{\mathbf L}}

\newcommand{\perf}{\operatorname{perf}}
\newcommand{\spa}{\operatorname{Spa}}
\newcommand{\Spec}{\operatorname{Spec}}
\newcommand{\Tor}{\operatorname{Tor}}
\newcommand{\ov}{\overline }
\newcommand{\ul}{\underline}

\begin{document}
\title{Regular rings and perfect(oid) algebras}
\dedicatory{Happy Birthday Gennady!}

\author[Bhatt]{Bhargav Bhatt}
\address{4832 East Hall, Department of Mathematics, University of Michigan, MI 48109, U.S.A. }
\email{bhargav.bhatt@gmail.com}

\author[Iyengar]{Srikanth B. Iyengar}
\address{Department of Mathematics, University of Utah, Salt Lake City, UT 84112, U.S.A.}
\email{iyengar@math.utah.edu}

\author[Ma]{Linquan Ma}
\address{Department of Mathematics, Purdue University, 150 N. University street, IN 47907}
\email{ma326@purdue.edu}

\thanks{During the preparation of this work, BB was partially supported by a Packard fellowship. We also thank the National Science Foundation for supporting us through grants DMS-1501461 (BB), DMS-1700985 (SBI), and DMS-1836867/1600198 and DMS-1252860/1501102. (LM)}

\date{\today}



\begin{abstract}
We prove a $p$-adic analog of Kunz's theorem: a $p$-adically complete noetherian ring is regular exactly when it admits a faithfully flat map to a perfectoid ring. This result is deduced from a more precise statement on detecting finiteness of projective dimension of finitely generated modules over noetherian rings via maps to perfectoid rings. We also establish a version of the $p$-adic Kunz's theorem where the flatness hypothesis is relaxed to almost flatness.
\end{abstract}


\maketitle

\section{Introduction}
This paper explores some homological properties of perfect(oid) algebras over commutative noetherian rings.  A commutative ring of positive characteristic $p$ is called perfect if its Frobenius endomorphism is an isomorphism. Perfectoid rings are generalizations of perfect rings to mixed characteristic (Definition~\ref{def:perfectoid}). Their most important features for our work are: if $A$ is a perfectoid ring, then $\sqrt{pA}$ is a flat ideal, $A/\sqrt{pA}$ is a perfect ring, and finitely generated radical ideals in $A$ containing $p$ have finite flat dimension  (Lemma~\ref{lem:PerfdFlatDim}).

One of our main results is that over a noetherian local ring $R$, any perfectoid $R$-algebra $A$ with $\fm A\ne A$  detects finiteness of homological dimension of $R$-modules. More precisely,  given such an $A$, if a finitely generated $R$-module $M$ satisfies $\Tor^{R}_{j}(A,M)=0$ for $j\gg 0$, then $M$ has a finite free resolution by $R$-modules (Theorem~\ref{thm:perfectoid}). The crucial property of $A$ that is responsible for this phenomenon is isolated in Theorem~\ref{thm:smoothable}, which identifies a large class of modules that can detect finiteness of homological dimension over local rings.

As a consequence, we obtain a mixed characteristic generalization of Kunz's theorem, resolving a question from \cite[Remark~5.5]{Bhatt:2016a}; see also \cite[pp.~6]{Andre:2018a}. Recall that Kunz's theorem asserts that a noetherian ring $R$ of characteristic $p$ is regular if and only if the Frobenius map $R \to R$ is flat. One can reformulate this result as the following assertion: such an $R$ is regular exactly when there exists a faithfully flat map $R \to A$ with $A$ perfect. Our $p$-adic generalization is the following:

\begin{Theorem}[see Theorem~\ref{thm:padicKunz}]
Let $R$ be a noetherian ring such that $p$ lies in the Jacobson radical of $R$ (for example, $R$ could be $p$-adically complete). Then $R$ is regular if and only if there exists a faithfully flat map $R \to A$ with $A$ perfectoid.
\end{Theorem}

Two algebras are of special interest: the absolute integral closure, $R^{+}$, of a domain $R$, and the perfection, $R_{\perf}$, of a local ring $R$ of positive characteristic. We prove that if $R$ is an excellent local domain of  positive characteristic and $\Tor^{R}_{j}(R^{+},k)=0$ or $\Tor^{R}_{j}(R_{\perf},k)=0$ for \emph{some} $j\ge 1$, then $R$ is regular; see Theorem~\ref{thm:regularity}, which contains also a statement about $R^{+}$ for local rings of mixed characteristic. A key input in its proof is that systems of parameters for $R$ are weakly proregular on $R^{+}$ and on $R_{\perf}$ (Lemma~\ref{lem:proregular}).

Over a perfectoid ring $A$ one has often to consider modules that are almost zero, meaning that they are annihilated by $\sqrt{pA}$. In particular, in the context of the theorem above, the more reasonable hypothesis on the map $R\to A$  is that it is almost flat, that is to say that $\Tor^{R}_{i}(-,A)$ is almost zero for $i\ge 1$.  With this in mind, in Section~\ref{sec:almost}, we establish the following, more natural, extension of the $p$-adic Kunz theorem in the almost setting.

\begin{Theorem}[see Corollary~\ref{cor:almostpadicKunz}]
Let $R$ be a noetherian $p$-torsionfree ring containing $p$ in its Jacobson radical.  If there exists a map $R \to A$ with $A$ perfectoid that is almost flat and zero is the only $R$-module $M$ for which $M\otimes_{R}A$  is zero, then $R$ is regular.
\end{Theorem}

\section{Criteria for finite flat dimension}
Let $S$ be a ring; throughout this work rings will be commutative but not always noetherian.  The flat dimension of an $S$-module $N$
is denoted $\fd_{S}(N)$; when $S$ is noetherian and $N$ is finitely generated, this coincides with the projective dimension of  $N$.  By a local ring $(R,\fm,k)$ we mean that $R$ is a ring with unique maximal ideal $\fm$ and residue field $k$.

\begin{theorem}
\label{thm:smoothable}
Let $(R,\fm,k)$ be a noetherian local ring and $S$ an $R$-algebra containing an ideal $J$ with $\fm S \subseteq J$ and $d\colonequals \fd_{S}(S/J)$  finite. Let $U$ be an $S$-module with $JU\ne U$.

If an $R$-module $M$ has $\Tor^{R}_{i}(U,M)=0$ for  $i=s,\dots, s+d$ and some integer $s\ge 0$, then
\[
\Tor^{R}_{s+d}(k,M)=0\,.
\]
In particular, if $\Tor^{R}_{j}(U,M)=0$ for $j\ge s$, then $\Tor^{R}_{j}(k,M)=0$ for $j\ge s+d$.
\end{theorem}

\begin{proof}
Set $V\colonequals (S/J)\lotimes_{S}U$, viewed as a complex of $S$-modules.  The hypothesis is that $\HH {i}{U\lotimes_{R}M}=0$ for $i=s,\dots, s+d$. Given the quasi-isomorphism of complexes
\[
V \lotimes_{R} M \simeq (S/J) \lotimes_{S} (U\lotimes_{R}M)
\]
it follows by, for example, a standard spectral sequence argument that $\HH {s+d}{V \lotimes_{R} M}=0$.  Since $\fm S\subseteq J$, the action of $R$ on $S/J$ and hence also on $V$, factors through $R/\fm$, that is to say, through $k$. Thus one has a quasi-isomorphism
\[
V \lotimes_{R} M \simeq V \lotimes_{k} (k\lotimes_{R}M)
\]
of complexes of $R$-modules.  The K\"unneth isomorphism then yields the isomorphism below
\[
0= \HH {s+d}{V \lotimes_{R} M} \cong \bigoplus_{j} \HH j{V} \otimes_{k} \HH {s+d-j}{k\lotimes_{R}M}\,.
\]
Since $\HH 0V\cong U/JU$ is nonzero, by hypotheses, it follows that $\HH{s+d}{k\lotimes_{R}M}=0$.
\end{proof}

\begin{corollary}
Let $S$ and $U$ be as in Theorem~\ref{thm:smoothable}. If a finitely generated $R$-module $M$ satisfies $\Tor^{R}_{i}(U,M)=0$ for  $i=s,\dots, s+d$ and some integer $s\ge 0$, then  $M$ has a finite free resolution of length at most $s+d$.
\end{corollary}

\begin{proof}
Since $R$ is a noetherian local ring, $\fd_{R}(M)=\sup\{i\mid  \Tor^{R}_{i}(k,M)\ne 0\}$ when $M$ is finitely generated. Thus the desired result is a direct consequence of Theorem~\ref{thm:smoothable}.
\end{proof}

A natural question that arises is whether $\Tor^{R}_{s}(U,M)=0$ for some $s\ge 0$ ensures that $M$ has a finite free resolution. This is indeed the case for  $M=k$; the argument depends on a rigidity result, recalled below.

\subsection{Rigidity}
\label{ss:rigidity}
Let $(R,\fm, k)$ be a noetherian local ring and $N$ an $R$-module. Set
\[
s(N) \colonequals \sup \{t\mid \lch{\fm}t{\Hom_R(N, E)}\neq 0\}\,.
\]
where $E$ is the injective hull of the $R$-module $k$ and $\lch{\fm}{*}{-}$ denotes local cohomology with respect to $\fm$. The following statements hold:
\begin{enumerate}[\quad\rm(1)]
\item
$s(N)\le \dim R$;
\item
If $\Tor_i^R(N,k)=0$ for some $i\geq s(N)$, then $\Tor_j^R(N,k)=0$ for all $j\geq i$.
\end{enumerate}
Indeed (1) holds because  $\lch{\fm}i{-}=0$ for $i>\dim R$; see~\cite[Theorem~3.5.7]{Bruns/Herzog:1998a}. Part (2) is contained in~\cite[Proposition~3.3]{Christensen/Iyengar/Marley:2018a} due to Christensen, Iyengar, and Marley.

\begin{corollary}
\label{cor:smoothable}
Let $S$ and $U$ be as in Theorem~\ref{thm:smoothable}.  If \, $\Tor^{R}_{i}(U,k)=0$ for some $i\geq\dim R$, then $R$ is regular.
\end{corollary}

\begin{proof}
By \ref{ss:rigidity} one has $\Tor_{j}^{R}(U, k)=0$ for all $j\ge i$.  Theorem \ref{thm:smoothable} then yields $\Tor_{j}^{R}(k,k)=0$ for $j\gg 0$, and so $R$ is regular; see by \cite[Theorem~2.2.7]{Bruns/Herzog:1998a}.
\end{proof}

The preceding result may be viewed as an generalization of the descent of regularity along homomorphisms of finite flat dimension. Indeed,  if $(R,\fm,k)\to (S,\fn,l)$ is a local homomorphism with $S$ regular and $\fd_{R}S$ finite, then Corollary~\ref{cor:smoothable} applies with $U\colonequals S$ and $J\colonequals \fn$ to yield that $R$ is regular.

\begin{remark}
\label{rem:AIM}
It should be clear from the proof of Theorem~\ref{thm:smoothable} that the ring $S$ in the hypothesis may well be a differential graded algebra, $U$ a dg $S$-module with $\hh{k\lotimes_{R}U}\ne 0$, and $M$ an $R$-complex; in that generality, the result compares with~\cite[Theorem~6.2.2]{Avramov/Iyengar/Miller:2006a}.
\end{remark}

\section{Recollections on perfect and perfectoid rings}

In this section, we recall the definition of perfect and perfectoid rings (with examples) and summarize their homological features most relevant to us. Fix a prime number $p$. In what follows $\bbZ_{p}$ denotes the $p$-adic completion of $\bbZ$.

\begin{subsection}{Perfect rings}
\label{ss:perfect}
Let $A$ be a commutative ring of positive characteristic $p$ and $\varphi\colon A\to A$ the Frobenius endomorphism: $\varphi(a)=a ^{p}$ for each $a\in A$. The ring $A$ \emph{perfect} if $\varphi$ is bijective; such an $A$ is reduced.

The \emph{perfect closure} of a ring $A$ is the colimit $A\xra{\varphi}A\xra{\varphi}\cdots$,  denoted $A_{\perf}$. It is easy to verify that $A_{\perf}$ is a perfect ring of characteristic $p$, and the map $A \to A_{\perf}$ is the universal map from $A$ to such a ring. Moreover the kernel of the canonical map $A\to A_{\perf}$ is precisely $\sqrt{0}$, the nilradical of $A$; in particular, when $A$ is reduced, we identify $A$ as a subring of $A_{\perf}$.

Each element $x$ in a perfect ring $A$ has a unique $p^{e}$-th root, for each $e\ge 1$. We set
\[
(x^{1/p^\infty}) \colonequals \bigcup_{e=1}^\infty(x^{1/p^e})A
\]
This ideal is reduced for it equals $\sqrt{xA}$.

\begin{lemma}
\label{lem:hochster}
Let $S$ be a perfect ring of positive characteristic. For any set $x_{1},\dots,x_{n}$ of elements of $S$, the ideal $J:=(x_{1}^{1/p^\infty},\dots, x_{n}^{1/p^\infty})$ of $S$ satisfies $\fd_{S}(S/J)\le n$. \qed
\end{lemma}

This result is due to Aberbach and Hochster~\cite[Theorem~3.1]{Aberbach/Hochster:1997a}. We recall an elementary proof from \cite[Lemma 3.16]{Bhatt/Scholze:2017a}.

\begin{proof}
If $R$ is any perfect ring of positive characteristic and $f \in R$, then $R/(f^{1/p^\infty})$ is reduced and hence also perfect. Thus, by induction, it is enough to treat the case when $n=1$. Relabel $x = x_1$ for visual convenience. It suffices to check that the ideal $I \colonequals (x^{1/p^\infty}) \subset S$ is flat as an $S$-module. Consider the direct limit
\[
M\colonequals   \varinjlim\Big(
S \xra{x^{1 - \frac{1}{p}}} S \xra{x^{\frac{1}{p} - \frac{1}{p^2}}} S \to \cdots \to S \xra{x^{\frac{1}{p^n} - \frac{1}{p^{n+1}}}} S \to \dots \Big).
\]
As $M$ is a direct limit of free $S$-modules, it is flat. There is a natural map $M \to I$ determined by sending $1 \in S$ in the $n$-th spot of the diagram above to $x^{\frac{1}{p^n}} \in I$. It is enough to show that this map is an isomorphism. Surjectivity holds as all generators of $I$ are hit.

As to injectivity, pick an element $\ov {g} \in M$ in the kernel. Then $\ov {g}$ lifts to an element $g \in S$ (viewed in the $n$-th copy of $S$ in the diagram above for some $n$) such that $g \cdot x^{\frac{1}{p^n}}=0$. But then  $(g \cdot x^{\frac{1}{p^{n+1}}})^{p} = 0$ and so $g \cdot x^{\frac{1}{p^{n+1}}} = 0$, as $S$ is reduced. As $\frac{1}{p^n} - \frac{1}{p^{n+1}} \geq \frac{1}{p^{n+1}}$, we get $g \cdot x^{\frac{1}{p^n} - \frac{1}{p^{n+1}}} = 0$, so $g$ is killed by the transition map in the system above, whence $\ov {g} = 0$.
\end{proof}

\begin{remark}
An important property of perfect rings is that they admit canonical (and unique) lifts to $\bbZ_p$. Indeed, if $A$ is a perfect ring of characteristic $p$, then the ring $W(A)$ of Witt vectors of $A$ is a $p$-torsionfree and $p$-adically complete ring with $W(A)/p \cong A$. In fact, any such lift is uniquely isomorphic to $W(A)$: the functors $A \mapsto W(A)$ and $B \mapsto B/p$ implement an equivalence between the categories of perfect rings of characteristic $p$ and the category of $p$-torsionfree and $p$-adically complete rings $B$ with $B/p$ being perfect. This perspective can help guess $W(A)$ in concrete situations. For instance, it follows that the $p$-adic completion of $\bbZ[x_1^{1/p^\infty},...,x_d^{1/p^\infty}]$ coincides with the Witt vectors of $\mathbb{F}_p[x_1,...,x_d]_{\perf}$.
\end{remark}

\end{subsection}

\begin{subsection}{Perfectoid rings}
\label{ss:perfectoid}

Let us recall the notion of a perfectoid ring from \cite[Definition 3.5]{Bhatt/Morrow/Scholze:2016a}; this is sometimes  referred to as \emph{integral perfectoid} to emphasize its integral nature, and to contrast it with the perfectoid Tate rings that arise in the context of perfectoid spaces.

For any commutative ring $A$, consider the {\em tilt} $A^\flat$ of $A$ defined by
\[
A^\flat\colonequals  \varprojlim_{x \mapsto x^p} A/p\,.
\]
This ring is perfect of characteristic $p$ and the projection map $A^\flat \to A/p$ is the universal map from a perfect ring to $A/p$.  Set
\[
A_{\inf}(A)\colonequals  W(A^\flat)\,.
\]
When $A$ is $p$-adically complete, the projection map $A^\flat \to A/p$ lifts uniquely to a map $\theta:A_{\inf}(A) \to A$, called \emph{Fontaine's $\theta$-map}.

\begin{definition}
\label{def:perfectoid}
A ring $A$ is \emph{perfectoid} if the following conditions hold:
\begin{enumerate}
\item $A$ is $p$-adically complete.
\item The Frobenius $A/p \to A/p$ is surjective.
\item The kernel of Fontaine's map $\theta:A_{\inf}(A) \to A$ is principal.
\item There exists an element $\varpi \in A$ with $\varpi^p = p u$ for a unit $u$.
\end{enumerate}
The category of perfectoid rings is the full subcategory of all commutative rings spanned by perfectoid rings.
\end{definition}

There is a more explicit characterization of perfectoid rings in terms of Teichmuller expansions that avoids directly contemplating the $A_{\inf}(-)$ construction. Recall that for any perfect ring $B$ of characteristic $p$, each $f \in W(B)$ can be written \emph{uniquely} as $\sum_{i=0}^\infty [b_i] p^i$ where $b_i \in B$ and the map $[\cdot]:B \to W(B)$ is the multiplicative (but not additive) Teichmuller section to the projection $W(B) \to B$;  we refer to this presentation as the Teichmuller expansion of $f$. One then has the following characterization of perfectoidness:

\begin{lemma}
\label{rmk:AltDefPerf}
A ring $A$ is perfectoid if and only if there exists a perfect ring $B$ of characteristic $p$ and an isomorphism $W(B)/(\xi) \cong A$, where the coefficient of $p$ in the Teichmuller expansion of $\xi$ is a unit (i.e. $\xi$ is primitive) and $B$ is $(\overline{\xi})$-adically complete, where $\overline{\xi}$ is the image of $\xi$ in $B$. For such a $B$, there is a unique identification $B \cong A^\flat$ compatible with the isomorphism $W(B)/(\xi) \cong A$. In particular, the element $\varpi$ appearing in Definition 3.5 can be assumed to admit a compatible system $\{ \varpi^{1/p^n} \}$ of $p$-power roots.
\end{lemma}

The equivalence of this definition with Definition~\ref{def:perfectoid} can be deduced from the discussion in \cite[\S 3]{Bhatt/Morrow/Scholze:2016a} (and is presumably present in \cite[\S 16]{Gabber/Ramero:2017}). The analogous characterization of perfectoid Tate rings can be found in \cite[Proposition 1.1]{Fontaine:2013}, \cite[Theorem 3.6.5]{Kedlaya/Liu:2015} and \cite[Theorem 3.17]{Scholze:2017}. An explicit reference is \cite[Remark 8.6]{Lau:2016a}. (We are grateful to Scholze for bringing \cite{Lau:2016a} to our attention.) For the convenience of the reader, we sketch the proof.

 \begin{proof}[Proof sketch]
 The ``only if'' direction is immediate from \cite[Remark 3.11]{Bhatt/Morrow/Scholze:2016a} as we can simply take $B = A^\flat$, so $W(B) = A_{\inf}(A)$. For the ``if'' direction, the formula $W(B)/(\xi) \cong A$ shows that $B/\overline{\xi} \cong A/p$, whence $B \cong A^\flat$ as $B$ is perfect and $(\overline{\xi})$-adically complete; the formula $A \cong W(B)/(\xi)$ can then be rewritten as $A \cong A_{\inf}(A)/(\xi)$, which immediately gives the perfectoidness of $A$.

 Finally, for $A = W(A^\flat)/(\xi)$, up to multiplication by units, we have $\xi = p - [a_0] u$ where $u \in W(A^\flat)^*$ and $a_0 \in A^\flat$. Set $\varpi \in A$ to be the image of $[a_0^{1/p}]$. Then $\varpi \in A$ satisfies (4) in Definition~\ref{def:perfectoid} and  admits a compatible system of $p$-power roots $\varpi^{1/p^n}\colonequals  [a_0^{1/p^{n+1}}]$.
\end{proof}

Perfectoid rings are reduced; this follows by combining \cite[Corollary 16.3.61 (i)]{Gabber/Ramero:2017} with \cite[Remark 3.8]{Bhatt/Morrow/Scholze:2016a}, or by arguing as in the first proof of \cite[Proposition 4.18 (3)]{Bhatt/Morrow/Scholze:2018}, or simply by \cite[Lemma 8.9]{Lau:2016a}. The most important feature of perfectoid rings for the purposes of this paper is that they are perfect modulo a flat ideal.

\begin{lemma}
\label{lem:PerfdFlatDim}
Let $A$ be a perfectoid ring. Then the ideal $\sqrt{pA} \subset A$ is flat and $\ov {A}\colonequals  A/\sqrt{pA}$ is a perfect ring of characteristic $p$. Moreover, if $J \subset A$ is any radical ideal containing $p$ such that $J\ov {A}$ generated by $n$ elements up to radicals, then $\fd_A(A/\sqrt{J}) \leq n+1$.
\end{lemma}

The construction $A \mapsto \ov {A}$ gives a functor from perfectoid rings to perfect rings, left adjoint to the inclusion in the other direction (see Example~\ref{ex:Perfectoid} (1) below).

\begin{proof}
Write $A = W(A^\flat)/(\xi)$ and let $\varpi$ be as in Lemma~\ref{rmk:AltDefPerf}, with a compatible system of $p$-power roots.
We first show that $A/(\varpi^{1/p^\infty})$ is perfect and thus reduced; since $(p) = (\varpi^{p})$, this will imply that $\sqrt{pA} = (\varpi^{1/p^\infty})$ and hence that $A/\sqrt{pA}$ is perfect. Note that
\[
A/(\varpi^{1/p^\infty}) \cong W(A^\flat)/(p - [a_0]u, [a_0^{1/p^\infty}]) \cong A^\flat/(a_0^{1/p^\infty}).
\]
This ring is perfect as $A^\flat$ is perfect, proving the claim.

Next, we check that $\fd_A(\ov {A}) \leq 1$; this is equivalent to showing that $\sqrt{pA}$ is a flat ideal. Note that if $A$ is $p$-torsionfree, then $\sqrt{pA} = (\varpi^{1/p^\infty})$ is a rising union of principal ideals generated by elements which are not zero divisors, and thus trivially flat. In general, as $\ov {A}$ has characteristic $p$, it is enough to check that $p$-complete Tor amplitude of $\ov {A} \in D(A)$ lies in $[-1,0]$, in the sense of \cite[\S 4.1]{Bhatt/Morrow/Scholze:2018}.  By \cite[Lemma 3.13]{Bhatt/Morrow/Scholze:2016a}, the square
\[ \xymatrix{ W(A^\flat) \ar[r] \ar[d] & A \ar[d] \\
		   W(\ov {A}) \ar[r] & \ov {A} }\]
is a Tor-independent pushout square. By base change, it is enough to check that $p$-complete Tor amplitude of $W(\ov {A}) \in D(W(A^\flat))$ lies in $[-1,0]$. As $p$ is not a zero divisor on both $W(A^\flat)$ and $W(\ov {A})$, this reduces to checking that the Tor amplitude of $\ov {A}$ over $A^\flat$ lies in $[-1,0]$. But this follows from Lemma~\ref{lem:hochster} since $\ov {A} = A^\flat/(a_0^{1/p^\infty})$.

The final assertion about flat dimensions now also follows from Lemma~\ref{lem:hochster} and transitivity of flat dimensions.
\end{proof}

\begin{example}
\label{ex:Perfectoid}
Let us give some examples relevant to this paper.

\subsubsection{\emph{(1)} Perfect rings} A ring $A$ of characteristic $p$ is perfectoid if and only if it is perfect. The ``if'' direction is immediate: the kernel of $A_{\inf}(A) \cong W(A) \to A$ is generated by $p$, and we may take $\varpi = 0$. For the reverse implication, see \cite[Example 3.11]{Morrow:2016}. In terms of Lemma~\ref{rmk:AltDefPerf}, these are exactly the perfectoid rings of the form $W(B)/(p)$ for a perfect ring $B$ of characteristic $p$.

\subsubsection{\emph{(2)} Absolute integral closures} If $A$ is an absolutely integrally closed domain, then its $p$-adic completion $\widehat{A}$ is perfectoid. This is clear from the previous example when $A$ has characteristic $p$. In the mixed characteristic case, the ring $A$ is a $p$-torsionfree ring that contains an element $\varpi \in A$ with $\varpi^p = p$. As $A$ is absolutely integrally closed, the $p$-power map map $A/(\varpi) \to A/(\varpi^p)$ is readily checked to be an isomorphism. But then the same holds true for $\widehat{A}/(\varpi) \to \widehat{A}/(\varpi^p)$. The perfectoidness of $\widehat{A}$ follows from \cite[Lemma 3.10]{Bhatt/Morrow/Scholze:2016a}.

\subsubsection{\emph{(3)} Perfectoid polynomial rings} The $p$-adic completion $A$ of $\bbZ[p^{1/p^\infty}, x_1^{1/p^\infty},\dots,x_d^{1/p^\infty}]$ is perfectoid. In terms of the characterization in Lemma~\ref{rmk:AltDefPerf}, one takes $B$ to be the $t$-adic completion of $\mathbb{F}_p [t, x_1,\dots,x_d]_{\perf}$ and $\xi = p - [t]$.

\subsubsection{\emph{(4)} Perfectoidification of unramified regular local rings} Let $(R,\fm,k)$ be a complete noetherian regular local ring of mixed characteristic. Assume $k$ is perfect and write $W = W(k)$ for the Witt vectors of $k$. Assume that $R$ is unramified, i.e., $p \notin \fm^2$. By the Cohen structure theorem, we can write $R \cong W \llbracket x_2,\dots,x_d \rrbracket$. Let $A$ be the $p$-adic completion of
\[
W[p^{1/p^\infty},x_2^{1/p^\infty}, \dots , x_d^{1/p^\infty}] \otimes_{W[x_1,\dots ,x_d]} R\,.
\]
By a variant of the previous example, one checks that the ring $A$ is perfectoid and the map $R \to A$ is faithfully flat.

\subsubsection{\emph{(5)} Perfectoidification of ramified regular local rings} Fix $(R,\fm,k)$ and $W$ be as in the first two sentences of the previous example. Assume that $R$ is ramified, i.e., $p \in \fm^2$. By the Cohen structure theorem, we can write $R = W \llbracket x_1,\dots, x_d \rrbracket/(p-f)$ where $f \in (x_1,\dots,x_d)^2$. Let $A$ be the $p$-adic completion of
\[
W[x_1^{1/p^\infty}, \dots , x_d^{1/p^\infty}] \otimes_{W[x_1,\dots ,x_d]} R\,.
\]
As observed by Shimomoto \cite[Proposition 4.9]{Shimomoto:2014}, the ring $A$ is perfectoid and the map $R \to A$ is faithfully flat (see also \cite[Example 3.4.5]{Andre:2016a}, \cite[Proposition 5.2]{Bhatt:2016a}). In terms of Lemma~\ref{rmk:AltDefPerf}, one takes $B$ to be the $(y_1,\dots,y_d)$-adic completion of $k[y_1,\dots,y_d]_{\perf}$ and $\xi = p-f([y_i])$ to get that $A \cong W(B)/(\xi)$ is perfectoid.

\subsubsection{\emph{(6)} Modifications in characteristic $p$} One can construct new perfectoid rings from old ones by changing their special fibers: if $A$ is a perfectoid ring, and $B \to \ov {A}$ is any map of perfect rings of characteristic $p$ (such as an inclusion) then $\widetilde{B}\colonequals  A \times_{\ov {A}} B$ is a perfectoid ring. Indeed, if we write $A = W(A^\flat)/(\xi)$ for a primitive element $\xi$, then setting $B'\colonequals  A^\flat \times_{\ov {A}} B$ gives a perfect ring of characteristic $p$ such that $W(B') \cong W(A^\flat) \times_{\ov {A}} B$. Thus. as $\xi$ maps to $0$ in $\ov {A}$, it lifts uniquely to $W(B')$. One  checks that $\widetilde{B} \cong W(B')/(\xi)$ is indeed perfectoid.

\subsubsection{\emph{(7)} Completed localizations} If $A$ is a perfectoid ring and $S \subset A$ is a multiplicative set, then the $p$-adic completion $\widehat{S^{-1} A}$ of $S^{-1} A$ is perfectoid. In particular, the $p$-adically completed local rings of $A$ are perfectoid.

\subsubsection{\emph{(8)} Products} An arbitrary product of perfectoid rings is perfectoid. This follows immediately from Definition~\ref{def:perfectoid} as the functor $A \mapsto A_{\inf}(A) = W(A^\flat) = W(\lim_{x \mapsto x^p} A/p)$ commutes with products.
\end{example}

\begin{remark}
All  examples above involve perfectoid rings that either have characteristic $p$ or were $p$-torsionfree. One can take products to obtain examples where neither of these properties holds true. Using Example~\ref{ex:Perfectoid} (6), one can also construct such examples which are closely related to products, but not themselves products. In fact, all examples  have this flavor: if $A$ is a perfectoid ring, then its maximal $p$-torsionfree quotient $A_{t\!f}\colonequals  A/A[p^\infty]$ is a perfectoid $p$-torsionfree ring, and the natural maps give an exact sequence
\[
0 \to A \to A_{t\!f} \times \ov {A} \to \ov {A_{t\!f}} \to 0
\]
of $A$-modules by \cite[Remark 8.8]{Lau:2016a}. This  realizes $A$ as the fiber product $A_{tf} \times_{\ov {A_{t\!f}}} \ov {A}$. In other words, one can construct arbitrary perfectoid rings by ``modifying'' $p$-torsionfree perfectoid rings by the procedure of Example~\ref{ex:Perfectoid} (6).
\end{remark}
\end{subsection}

\section{Applications}
In this section we prove the results dealing with perfect, and with perfectoid, algebras stated in the Introduction. Fix a prime $p$.

\begin{theorem}
\label{thm:perfectoid}
Let $(R,\fm,k)$ be a noetherian local ring and $A$ an $R$-algebra with $\fm A\ne A$. Assume that $A$ is perfectoid (and thus $k$ has characteristic $p$). Set $n = \dim(R)$.

If $M$ is an $R$-module with $\Tor^{R}_{i}(A,M)=0$ for  $s\le i\le  s+n+1$ and integer $s\ge 0$, then
\[
\Tor_{s+n+1}^{R}(k,M)=0\,.
\]
In particular, if $\Tor^{R}_{j}(A,M)=0$ for $j\gg 0$, then $\Tor^{R}_{j}(k,M)=0$ for $j\gg 0$.
\end{theorem}

\begin{proof}
Let $J \subset A$ the radical of the ideal generated by a system of parameters $x_1,...,x_n$ of $R$. Then $\fm A \subset J$. Moreover, since $J$ is generated by $n$ elements up to radicals and contains $p$,  Lemma~\ref{lem:PerfdFlatDim} implies that $\fd_A(A/J) \leq n+1$. Theorem~\ref{thm:smoothable} applied with $S = A$, $U = A$ and $d = n+1$ then implies the result.
\end{proof}

\begin{remark}
\label{rem:oneless}
If the perfectoid ring $A$ appearing in Theorem~\ref{thm:perfectoid} is either perfect or $p$-torsionfree, the hypothesis on the vanishing range can be improved slightly: it is sufficient to require vanishing of $\Tor^{R}_{i}(A,M)=0$ when  $s\le i\le  s+n$ for an integer $s\ge 0$.

To see this, note that we can replace the sequence $x_1,...x_n$ in the proof above by a system of parameters for the image $\ov {R}$ of $R \to A$ without changing the proof. Now if $A$ is perfect, then $\ov {R}$ has characteristic $p$, so we can replace the reference to Lemma~\ref{lem:PerfdFlatDim} by a reference to Lemma~\ref{lem:hochster}, which results in an improvement of $1$ in the range of vanishing. If $A$ is $p$-torsionfree, so is $\ov {R}$, so we may take $x_1 = p$, in which case the ideal $J\ov {A}$ appearing in the proof above is generated by $n-1$ elements, up to radical; this again results in an improvement of $1$. We do not know if this improvement of $1$ is possible in general.
\end{remark}

\begin{remark}
\label{rem:perfectoid}
It will be evident from the proof that the result above holds under the weaker hypothesis that the $R$-algebra $A$ factors through a perfectoid algebra. What is more, the only relevant property of perfectoid algebras that is needed is that it is of positive characteristic and perfect modulo a flat ideal.
\end{remark}

There is an extension of Theorem~\ref{thm:perfectoid} to the non-local case. It can be viewed an extension of Herzog's result~\cite[Satz~3.1]{Herzog:1974a} that, for a finitely generated $R$-module $M$, vanishing of Tor against high Frobenius twists of $R$ implies finiteness of the flat dimension of $M$.

\begin{corollary}
\label{cor:flat-dim}
Let $R$ be a noetherian ring with $p$ in its Jacobson radical and let $A$ be a perfectoid $R$-algebra such that $\Spec A\to\Spec R$ is surjective. Let $M$ be an $R$-module satisfying $\Tor_{i}^{R}(A, M)=0$ for $i\gg 0$.

If $M$ is finitely generated, or $p=0$ in $R$ and $\dim R$ is finite,  then $\fd_RM$ is finite.
\end{corollary}

\begin{proof}
Fix a prime $\fp\subset  R$ containing $p$. The hypothesis implies $\Tor_i^{R_\fp}(A_{\fp}, M_{\fp})=0$ for $i\gg 0$. Then $\fp A_{\fp} \neq A_{\fp}$, because $\Spec A\to\Spec R$ is surjective.  The ring $A_{\fp}$ is perfect modulo a flat ideal (since the same was true for $A$), so Theorem~\ref{thm:perfectoid} (see also Remark~\ref{rem:perfectoid}) applies and yields that, with $k(\fp)$ the residue of the local ring $R_{\fp}$, one has
\begin{equation}
\label{eq:tor-vanishing}
\Tor_{i}^{R_{\fp}}(k(\fp), M_{\fp})=0\quad\text{for all $i\gg 0$}.
\end{equation}
Since $p$ is in the Jacobson radical, this conclusion holds whenever $\fp$ is a maximal ideal.

When the $R$-module $M$ is finitely generated, the $R_{\fp}$-module $M_{\fp}$ is finitely generated and the vanishing condition above implies that $\fd_{R_{\fp}}(M_{\fp})$ is finite, and since $\fp$ can be an arbitrary maximal ideal, it follows from a result of Bass and Murthy~\cite[Lemma~4.5]{Bass/Murthy:1967a} that $\fd_{R}M$ is finite.

Assume $p=0$ in $R$ and that $\dim R$ is finite. In particular \eqref{eq:tor-vanishing} holds for each prime $\fp$ in $R$. It follows from \cite[Theorem~4.1]{Christensen/Iyengar/Marley:2018a} that
\[
 \Tor_{i}^{R_{\fp}}(k(\fp), M_{\fp})=0 \quad\text{for $i\ge \dim R+1$}
\]
and hence, again from \emph{op.\ cit.},  that $\fd_{R}M$ is finite.
\end{proof}

\begin{remark}
In Corollary~\ref{cor:flat-dim}, the additional hypotheses on $R$, or on $M$, is necessary. For example, if $R$ is a regular ring of positive characteristic and $M:=\oplus_{\fm\in \mathrm{Max}\, R}(R/\fm)$,  one has $\Tor^{R}_{i}(R_{\perf},M)=0$ for each $i\ge 1$, since $R_{\perf}$ is a flat $R$-module. However one has $\fd_{R} M =\dim R$, and the latter can be infinite; see \cite[Appendix A1]{Nagata:1975a}.
\end{remark}

\subsection{Regularity}
\label{ss:regularity}
We can now prove a mixed characteristic generalization of Kunz's theorem~\cite[Theorem~2.1]{Kunz:1969a}, answering a question in \cite[Remark~5.5]{Bhatt:2016a}; see also~\cite[pp. 6]{Andre:2018a}.

\begin{theorem}
\label{thm:padicKunz}
Let $R$ be a noetherian ring with $p$ in its Jacobson radical.  If $R$ is regular, then there exists a faithfully flat map $R \to A$ with $A$ perfectoid.

Conversely, fix a map $R \to A$ with $A$ perfectoid. If\, $\Spec A\to\Spec R$ is surjective and $\Tor^{R}_{i}(A,A)=0$ for $i\gg 0$ (for example, if $A$ is a faithfully flat $R$-algebra), then $R$ is regular.
\end{theorem}

\begin{proof}
Assume first that there exists an $R$-algebra $A$ with the stated properties. Fix a maximal ideal $\fm$ of $R$. Since $p\in \fm$ holds, by hypothesis,  arguing as in the proof of Corollary~\ref{cor:flat-dim} (with $M=A$) one gets the vanishing of Tor below:
\[
\Tor_{i}^{R}(k(\fm), A) \cong \Tor_{i}^{R_{\fm}}(k(\fm), A_{\fm})=0\quad\text{for all $i\gg 0$}.
\]
The isomorphism holds because the module on the left is $\fm$-local. Since $k(\fm)=R/\fm$, it is a finitely generated $R$-module, so another application of Corollary~\ref{cor:flat-dim}, now with $M=k(\fm)$, implies that $\fd_{R}k(\fm)$ is finite, and hence that $R_{\fm}$ is regular. Since this holds for each maximal ideal of $R$, one deduces that $R$ is regular.

Conversely, assume that $R$ is regular with $p \in J(R)$. We must construct a faithfully flat map $R \to A$ with $A$ perfectoid. It clearly suffices to do this on each connected component of $\mathrm{Spec}(R)$, so we may assume $R$ is a regular domain. When $R$ has characteristic $p$, we may simply take $A = R_{\perf}$ by Kunz's theorem. Assume from now on that $R$ is $p$-torsionfree.

Let us first explain how to construct the required cover when $R$ is complete noetherian regular local ring of mixed characteristic $(0,p)$. By gonflement~\cite[IX, App., Theorem 1, Cor.]{Bourbaki:CA9}, we may assume that the residue field of $R$ is perfect. One can then perform the constructions in Examples~\ref{ex:Perfectoid} (4) and (5) to obtain the required covers.

It remains to globalize. As $R$ is noetherian and $p \in J(R)$,  each completion $\widehat{R_{\fm}}$ at a maximal ideal $\fm$ is flat over $R$ and has mixed characteristic $(0,p)$. By the previous paragraph, for each maximal ideal $\fm \subset R$, we may choose a faithfully flat map  $\widehat{R_{\fm}} \to A(\fm)$ with perfectoid target. Consider the resulting map
\[
R \to \prod_{\fm} \widehat{R_{\fm}} \to \prod_{\fm} A(\fm)\,.
\]
As $R$ is noetherian, an arbitrary product of flat $R$-modules is flat, so the above map is flat. Moreover, it is also faithfully flat: the image of the induced map on $\mathrm{Spec}(-)$ is generalizing (by flatness) and hits all closed points (by construction), and hence must be everything. As a product of perfectoid rings is perfectoid, we have constructed the desired covers.
\end{proof}

In fact, it is possible to give a more precise characterization of regularity in terms of perfectoids than that given in Theorem~\ref{thm:padicKunz}. For instance, it is enough to assume a single $\mathrm{Tor}$-module vanishes in a sufficiently large degree.

\begin{corollary}
\label{cor:regularity}
Let $R$ be a noetherian local ring with residue field $k$, and let $A$ be an $R$-algebra with $\fm A\ne A$. Assume $A$ is perfectoid. If\, $\Tor^{R}_{i}(A, k)=0$ for some integer $i\geq\dim R$ (for example, if $\fd_{R}A$ is finite), then $R$ is regular.
\end{corollary}

\begin{proof}
The hypothesis implies $\Tor^{R}_{j}(A,k)=0$ for $j\ge i$; this is by \ref{ss:rigidity}. Then Theorem~\ref{thm:perfectoid} implies $\Tor^{R}_{j}(k,k)=0$ for $j\gg 0$, so $R$ is regular, by \cite[Theorem~2.2.7]{Bruns/Herzog:1998a}.
\end{proof}

In the preceding result, we do not know if the requirement that $i\geq \dim R$ is necessary. Next we prove that, for special $A$, it is not. To this end we recall the notion of a proregular sequence introduced by Greenlees and May~\cite{Greenlees/May:1992a}. The treatment due by Schenzel~\cite{Schenzel:2003a} is better suited to our needs.

\subsection{Proregular sequences}
\label{ss:proregular}
A sequence of elements $\ul x\colonequals x_{1},\dots,x_{d}$ in a ring $S$ is \emph{proregular} if for  $i\colonequals 1,\dots, d$ and integer $m\ge 1$, there exists an integer $n\ge m$ such that
\[
((x_{1}^{n},\dots, x_{i-1}^{n}):_{S}\, x_{i}^{n})\subseteq ((x_{1}^{m},\dots, x_{i-1}^{m}):_{S}\, x_{i}^{n-m})
\]
It is not hard to verify that this property holds if $\ul x$ is a regular sequence, or if $S$ is noetherian; see~\cite[pp.~167]{Schenzel:2003a}. By \cite[Lemma~2.7]{Schenzel:2003a} such a sequence is \emph{weakly proregular}, that is to say, for each $m$, there exists an integer $n\ge m$ such that the canonical map
\[
\HH i{x_{1}^{n},\dots, x_{d}^{n}; S}\lra \HH i{x_{1}^{m},\dots, x_{d}^{m}; S}
\]
on Koszul homology modules is zero for $i\ge 1$.

We care about these notions because of the following observation.

\begin{lemma}
\label{lem:proregular}
Let  $\fa$ be an ideal in a noetherian ring $R$ and $S$ an $R$-algebra.  If $\fa$ can be generated, up to radical, by a sequence whose image in $S$ is weakly progregular, then $\lch{\fa}iI =0$ for $i\ge 1$ and any injective $S$-module $I$.
\end{lemma}

\begin{proof}
By hypothesis, there exists a sequence $\ul x$ in $R$ such that $\sqrt{\ul x}=\sqrt {\fa}$ and ${\ul x}S$, the image of the sequence $\ul x$ in $S$, is weakly proregular. Let $C$ be the \v{C}ech complex on $\ul x$, so that $\lch{\fa}iM =\HH i{C\otimes_{R}M}$ for any $R$-module $M$; see, for example, \cite[Theorem~3.5.6]{Bruns/Herzog:1998a}. Since $I$ is an $S$-module has one $C\otimes_{R}I\cong (C\otimes_{R}S)\otimes_{S}I$. Since $C\otimes_{R}S$ is the \v{C}ech complex on ${\ul x}S$, the desired result then follows from ~\cite[Theorem~3.2]{Schenzel:2003a}.
\end{proof}

\subsubsection{Absolute integral closure}
Given a domain $R$, its absolute integral closure (that is to say, the its integral closure in an algebraic closure of it field of fractions), is denoted $R^{+}$. When $R$ is of positive characteristic,  $R^{+}$ contains a subalgebra isomorphic to $R_{\perf}$.

When $R$ has mixed characteristic, with residual characteristic $p$, the ideal $(p^{1/p^{\infty}})R^{+}$ is flat, and the quotient  ring $R^{+}/(p^{1/p^{\infty}})R^{+}$ is of characteristic $p$ and perfect. In the light of Remark~\ref{rem:perfectoid}, it follows that the conclusion of Corollary~\ref{cor:regularity} also holds when $R$ has mixed characteristic and the $R$-algebra $A$ factors through $R^{+}$.

\begin{proposition}
\label{prp:proregular}
Let $\ul x$ be a system of parameters in an excellent local domain $R$.

If $R$ has positive characteristic, then $\ul x$ is weakly proregular in $R_{\perf}$ and in $R^{+}$.

If $R$ has mixed characteristic and $\dim R\leq 3$, then $\ul x$ is weakly proregular in $R^{+}$.
\end{proposition}

\begin{proof}
By \cite[Corollary~3.3]{Schenzel:2003a}, it suffices to verify that there is some choice of an s.o.p. that is (weakly) proregular  on $R_{\perf}$, or $R^{+}$, as the case maybe.

We treat first the case where $R$ has positive characteristic. In this case, $R^{+}$ is a balanced big Cohen-Macaulay algebra, as proved by Hochster and Huneke~\cite[Theorem~1.1]{Hochster/Huneke:1992a}; see also Huneke and Lyubeznik~\cite[Corollary~2.3]{Huneke/Lyubeznik:2007a}. Thus any s.o.p. for $R$, in particular, $\ul x$ is a regular sequence, and hence also a  (weakly) proregular sequence, in $R^+$.

As to $R_{\perf}$: Since $R$ is excellent, it is a homomorphic image of an excellent Cohen-Macaulay local ring \cite[Corollary 1.2]{Kawasaki:2002a} and hence it admits a $p$-standard s.o.p.~\cite[Definition 2.1 and Theorem 1.3]{Cuong/Cuong:2017a}. In particular, $R$ admits an s.o.p.  $\ul x\colonequals  x_1,\dots,x_d$ such that
\[
((x_1^{n_1},\dots,x_{i-1}^{n_{i-1}}):_{R}\, x_i^{n_i}x_j^{n_j})= ((x_1^{n_1},\dots,x_{i-1}^{n_{i-1}}):_{R}x_j^{n_j})
\qquad\text{for all $1\leq i\leq j\leq d$.}
\]
In other words,  $\ul x$ is a strong $d$-sequence; see \cite[Definition~5.10]{Huneke:1996a}. We claim that such an $\ul x$ satisfies
\[
((x_1^n,\dots,x_{i-1}^n):_{R_{\perf}}\, x_i^{n}) \subseteq ((x_1^m,\dots,x_{i-1}^m):_{R_{\perf}}\, x_i^{n-m})
\qquad \text{for all $m\ge 1$ and $n>m$.}
\]
To this end, since $R_{\perf}=\cup\, R^{1/p^e}$, it suffices to prove that
\[
((x_1^{p^en},\dots,x_{i-1}^{p^en}):_{R}\, x_i^{p^en}) \subseteq ((x_1^{p^em},\dots,x_{i-1}^{p^em}):_{R}\, x_i^{p^en-p^em}).
\]
But the conditions on $\underline{x}$ imply
\begin{align*}
((x_1^{p^en},\dots,x_{i-1}^{p^en}):_{R}\, x_i^{p^en})
	& =((x_1^{p^en},\dots,x_{i-1}^{p^en}):_{R}\, x_i) \\
	&\subseteq ((x_1^{p^em},\dots,x_{i-1}^{p^em}):_{R}x_i)\\
	&=((x_1^{p^em},\dots,x_{i-1}^{p^em}):_{R}x_i^{p^en-p^em})
\end{align*}
This proves $\underline{x}$ is proregular, and hence also weakly proregular, on $R_{\perf}$, as desired.

Suppose $R$ is of mixed characteristic. When $\dim R\le 2$ once again $R^{+}$ is a big Cohen-Macaulay algebra, so the result follows.
Assume $\dim R=3$ and choose an s.o.p.\ of the form $x,y,p$. Since $R^{+}$ is normal, $x,y$ is a regular sequence on it. It thus suffices to verify:
\[
((x^n, y^n)\colon_{R^+}p^n) \subseteq ((x^m, y^m)\colon_{R^+}p^{n-m})\quad\text{for all $m\ge 1$ and $n>m$}.
\]
Assume $zp^n\in (x^n, y^n)R^+$. It follows from a result of Heitmann's~\cite[Theorem~0.1]{Heitmann:2005a}, see also \cite{Heitmann:2002a}, that  $p^\epsilon z\in (x^{n},y^{n})R^{+}$ for \emph{any} rational number $\epsilon$. In particular, $pz\in (x^n,y^n)R^+$ and this implies the desired inclusion.
\end{proof}

\begin{corollary}
\label{cor:proregular}
Let $(R,\fm,k)$ be an excellent local domain and $E$ the injective hull of $k$.
\item
If $R$ has positive characteristic, then
\begin{gather*}
 \lch{\fm}i{\Hom_R(R^+, E)}=0=\lch {\fm}i{\Hom_R(R_{\perf}, E)} \qquad \text{for each $i\ge 1$.}\\
\intertext{When $R$ has mixed characteristic and $\dim R\leq 3$, one has}
\lch {\fm}i{\Hom_R(R^{+}, E)}=0 \qquad \text{for each $i\ge 1$.}
\end{gather*}
\end{corollary}

\begin{proof}
By adjunction, the $R^{+}$-module $\Hom_R(R^+, E)$  and the $R_{\perf}$-module $\Hom_R(R_{\perf}, E)$ are injective. Thus the desired result follows from Proposition~\ref{prp:proregular} and Lemma~\ref{lem:proregular}.
\end{proof}

Aberbach and Li~\cite[Corollary~3.5]{Aberbach/Li:2008a} have proved parts (1) and (2) of the following result, using different methods.

\begin{theorem}
\label{thm:regularity}
Let $(R,\fm,k)$ be an excellent local domain. Then $R$ is regular if any one of the following conditions.
\begin{enumerate}[\quad\rm(1)]
\item
$R$ has positive characteristic and  $\Tor_{i}^{R}(R_{\perf},k)=0$ for some integer $i\ge 1$;
\item
$R$ has positive characteristic and $\Tor_{i}^{R}(R^{+},k)=0$ for some integer $i\ge 1$;
\item
$R$ has mixed characteristic, $\dim R\leq 3$, and $\Tor_{i}^{R}(R^{+},k)=0$ for some $i\ge 1$.
\end{enumerate}
\end{theorem}

\begin{proof}
In all cases, it follows from Corollary \ref{cor:proregular} and  \ref{ss:rigidity} that $\Tor_{j}^{R}(R^{+},k)$, respectively, $\Tor_{j}^{R}(R_{\perf},k)$, is zero for each $j\ge i$.  When $R$ has positive characteristic, $R^{+}$ contains $R_{\perf}$; in mixed characteristic, $R^{+}$ is perfect modulo a flat ideal. Therefore, in either case  Theorem~\ref{thm:perfectoid}---see also Remark~\ref{rem:perfectoid}---implies $\Tor_{j}^{R}(k,k)=0$ for $j \gg 0$ as desired.
\end{proof}

Here is a question suggested by part (3) above: If  $(R,\fm, k)$ is a noetherian local domain of characteristic $0$ and $\Tor_{i}^{R}(R^{+},k)=0$ for some $i\ge 1$, then is $R$ regular?

\section{Almost flatness}
\label{sec:almost}
The goal of this section is to prove, for rings of mixed characteristic, the variations of Theorems~\ref{thm:perfectoid} and ~\ref{thm:padicKunz} where the vanishing of Tor and the flatness hypotheses are relaxed to almost conditions. As before, throughout this section we fix a prime $p$; the notion of perfectoid is with respect to this prime. A module over a perfectoid ring $A$ \emph{almost zero} if it is killed by $\sqrt{pA}$; a map $R\to A$ is \emph{almost flat} if $\mathrm{Tor}_i^R(-,A)$ is almost zero for each $i \ge 1$.

In what follows we will consider maps $R\to A$ with $R$ noetherian and $p$-torsion free, and $A$ perfectoid, satisfying the following:

\subsection{Valuative condition}
\label{ss:valuative}
For every map $R \to V$ with $V$ a $p$-torsionfree and $p$-adically complete  rank $1$ valuation ring, there exists an extension $V \to W$ of $p$-torsionfree and $p$-adically complete  rank $1$ valuation rings and a map $A \to W$ extending $R \to V \to W$.

See Remark~\ref{rem:generic-points} for an alternative description of this condition, and Proposition~\ref{prp:AFFSpa} for a sufficient, and perhaps easier to verify, condition under which it holds.

\medskip

Compare the result below with Theorem~\ref{thm:perfectoid}, and also Remark~\ref{rem:oneless}.

\begin{theorem}
\label{thm:almostperfectoid}
Let $(R,\fm,k)$ be a noetherian local ring that is $p$-torsionfree and $p\in \fm$ holds. Let $R \to A$ be a map with $A$ perfectoid and satisfying the valuative condition~\ref{ss:valuative}.

If $M$ is an $R$-module for which the $A$-module $\Tor^{R}_{i}(A,M)$ is almost zero for  each integer $s\le i\le  s+\dim R$, for some  $s\ge 0$, then
\[
\Tor^{R}_{s+\dim R+1}(k,M)=0\,.
\]
In particular, if $M=k$, then the ring $R$ is regular.
\end{theorem}

The proof of this result, given further below, is a little more involved than that of Theorem~\ref{thm:perfectoid}. It will be clear from the proof that there is a version of the preceding result where almost zero is measured with respect to some fixed nonzero divisor in $A$, not necessarily $p$, that admits a compatible system of $p$-power roots. Moreover, it suffices that condition~\ref{ss:valuative} holds for noetherian valuations $V$ that dominate the maximal ideal $\fm$.

Here is an analogue of Theorem~\ref{thm:padicKunz}.

\begin{theorem}
\label{thm:almostpadicKunz}
Let $R$ be a noetherian $p$-torsionfree ring such that $p$ lies in its Jacobson radical. Let $R \to A$ be a map with $A$ perfectoid and satisfying the valuative condition~\ref{ss:valuative}. If $R\to A$ is almost flat, then $R$ is regular.
\end{theorem}

\begin{proof}
Fix a maximal ideal $\fm$ of $R$. We shall prove that $R_{\fm}$ is regular; as $\fm$ was arbitrary, the theorem follows. Since $p$ is in the Jacobson radical of $R$, the residue field $k$ at $\fm$ has characteristic $p$. Let $\widehat{A_{\fm}}$ denote the $p$-adic completion of $A_{\fm}$; this is a perfectoid ring. It is easy to verify that the valuative condition~\ref{ss:valuative} is inherited by the induced map $R_{\fm}\to \widehat{A_{\fm}}$. Since $k$ is of characteristic $p$ and is $\fm$-local, for each $i$ there are natural isomorphisms
\[
\Tor^{R_{\fm}}_{i}(\widehat{A_{\fm}},k)\cong \Tor^{R_{\fm}}_{i}(A_{\fm},k) \cong \Tor^{R}_{i}(A,k)\,.
\]
Since $R\to A$ is almost flat, it thus follows that the $\widehat{A_{\fm}}$-module $\Tor^{R_{\fm}}_{i}(\widehat{A_{\fm}},k)$ is almost zero for $i\ge 1$. Thus, Theorem~\ref{thm:almostperfectoid} applies and yields that $R_{\fm}$ is regular, as desired.
\end{proof}

Observe that, in contrast with the statement of Theorem~\ref{thm:padicKunz}, the preceding result makes no explicit hypothesis on the induced map of spectra of $R$ and $A$. But in fact the valuative condition~\ref{ss:valuative} can be described in terms of adic spectra.

\begin{remark}
\label{rem:generic-points}
Give a $p$-torsion free commutative ring $B$, let $\spa(B[1/p],B)$ denote the adic spectrum of $(B[1/p],B)$ topologized using the $p$-adic topology on $B$; see~Huber \cite[Definition (iii)]{Huber:1993a} and also \cite[\S10.3]{Conrad:2014a}, keeping in mind that $\spa(B[1/p],B)$ coincides with $\spa(B[1/p],B^{+})$, where $B^{+}$ is the integral closure of $B$ in $B[1/p]$. The generic points of $\spa(B[1/p],B)$ are in bijective correspondence with equivalence classes of maps $B \to V$ where $V$ is a $p$-torsionfree and $p$-adically complete rank $1$ valuation ring; the equivalence relation is generated by refinements of such $V$.

The valuative condition~\ref{ss:valuative} is thus equivalent to the surjectivity on generic points of the induced map $\spa(A[1/p],A) \to \spa(R[1/p],R)$. For psychological ease, we  remark that if a generic point $x \in \spa(R[1/p], R)$ is the image of a point $y \in \spa(A[1/p],A)$, then we can also find a generic point $y' \in \spa(A[1/p],A)$ lifting $x$ simply by setting $y'$ to be the maximal generalization of $y$.
\end{remark}

\begin{remark}
The main reason to use the valuative condition~\ref{ss:valuative} in formulating Theorem~\ref{thm:almostperfectoid} is that nonzero finitely generated ideals in a valuation ring cannot contain elements of arbitrarily small valuation. This provides an easy way to test whether certain modules not almost zero (which is a {\em stronger} statement than merely requiring them to be nonzero); see the paragraph following Claim 1 in the proof of Theorem~\ref{thm:almostperfectoid}. The restriction to rank $1$ valuations ensures that we may replace the ring $A$ appearing in the statement of Theorem~\ref{thm:almostperfectoid} with an almost isomorphic one without affecting the hypotheses on $A$.
\end{remark}


\begin{proof}[Proof of Theorem~\ref{thm:almostperfectoid}]
The assumptions on $R \to A$ are stable under replacing $A$ with an almost isomorphic perfectoid ring (such as $A_{t\!f}$). Thus, we may assume $A$ is $p$-torsionfree.

Set $d:=\dim R$ and choose elements $f_1,...,f_d$ in $\fm$ that generate it up to radical.

We begin by replacing the images of the $f_i$'s in $A$ by elements that admit $p$-power roots as follows: Choose  $g_1 = \varpi^{p}$, so that $(g_1) = (p) = (f_1)$; see Lemma~\ref{rmk:AltDefPerf}.  For $i \ge 2$, choose elements $h_i \in A^\flat$ lifting $f_i \in A/(g_1)$ and set $g_i = h_i^\sharp$. Then each $g_i$ admits a compatible system $\{g_i^{1/p^n}\}$ of $p$-power roots. Moreover,  by construction we have
\[
\big(g_i^{1/p^n}\big)^{p^n} \equiv f_i \mod (g_1) \quad\text{for $i \ge 2$.}
\]
In particular, there is an equality  $(f_1,...,f_d) = (g_1,...,g_d)$ of ideals of $A$.

The key step will be to justify the following

\begin{claim}
\label{cl:key}
When $\Tor^{R}_{s+d+1}(k,M)\ne 0$ holds,  there is an containment of ideals
\begin{equation}
\label{eq:badcontainment}
(g_1^{1/p^\infty}) \subseteq (g_{1},...,g_{d})A = (f_{1},...,f_{d})A \subset A.
\end{equation}
\end{claim}

Given this we complete the proof  by checking that \eqref{eq:badcontainment} is not compatible with the valuative condition. Choose a map $A \to W$ to a $p$-adically complete and $p$-torsionfree rank $1$ valuation ring $W$ such that the image of $f_i$ in $W$ is not invertible; to construct such a map, one first does it for $R$---where it exists since $R$ is $p$-torsionfree and $(f_i)$ is not the unit ideal~\cite[Theorem~6.4.3]{Huneke/Swanson:2006a}---and then invokes condition \ref{ss:valuative}.  As $(g_1) = (p)$ and $p$ is a pseudouniformizer in $W$, elements of $(g_1^{1/p^\infty})$ give elements of $W$ with arbitrarily small valuation. On the other hand, the ideal $(f_{1},...,f_{d}) \subseteq W$ is finitely generated and non-unital by construction, so it cannot contain elements of arbitrarily small valuation. In particular, it cannot contain $(g_1^{1/p^\infty})$,  contradicting \eqref{eq:badcontainment}. See  Proposition~\ref{prp:AFFSpa} for an alternative denouement.

\medskip

Now we take up the task of proving Claim~\ref{cl:key}. To that end for each integer $n\ge 1$ set
\[
A_{n}:= \mathrm{K}(g_{1}^{1/p^n},...,g_{d}^{1/p^n};\, A)\,,
\]
the Koszul complex over $A$ on the elements  $g_{1}^{1/p^n},...,g_{d}^{1/p^n}$.

\begin{claim}
\label{cl:almostzero}
For each $n$ the $A$-module $\Tor_{s+d+1}^{R}(A_{n},M)$ is almost zero.
\end{claim}

This is a straightforward verification using the fact that $A_{n}$ can be constructed as an iterated mapping cone ($d$ of them are required) starting with $A$,  and our hypothesis that $\Tor_{i}^{R}(A,M)$ is almost zero for $s\le i\le s+d$.

\medskip

In the next steps we will exploit the fact that each $A_{n}$ has a structure of a strict graded-commutative dg (differential graded) $A$-algebra; namely, it is an exterior algebra over $A$ on indeterminates $y_{n,1},\dots, y_{n,d}$ of degree one with differential defined by the assignment $y_{n,i}\mapsto g_{i}^{1/p^n}$. For each integer $n\ge 1$, writing one gets a morphism of dg $A$-algebras
\[
 A_{n}\to A_{n+1} \quad\text{where}\quad   y_{n,i} \mapsto  (g_{i}^{\frac 1{p^{n}} - \frac 1{p^{n+1}}}) y_{n+1,i}\,.
\]
Then $A_{\infty}:=\mathrm{colim}_{n}\, A_{n}$ is a strict graded-commutative dg $A$-algebra, and the structure maps $A_{n}\to A_{\infty}$ are morphisms of dg $A$-algebras.

\begin{claim}
\label{cl:acyclic}
The dg $A$-algebra $A_{\infty}$  satisfies
\[
 \HH i{A_{\infty}}=
 \begin{cases}
A/(g_{1}^{1/p^\infty},\dots, g_{d}^{1/p^\infty}) & \text{for $i=0$} \\
 0  &\text{for $i\ge 1$.}
 \end{cases}
 \]
\end{claim}

To begin with, set $B:=A/(\sqrt{pA})$; this ring is of characteristic $p$ and is perfect; see Lemma~\ref{lem:PerfdFlatDim}. Since $A$ is $p$-torsion free, the dg $A$-algebra
\[
0\to A\xra{p^{1/p^{n}}} A\to 0
\]
is quasi-isomorphic to its homology module in degree zero, namely, $A/(p^{1/p^{n}})$. Thus the colimit, as $n\to \infty$, of these dg algebras is quasi-isomorphic to $A/(p^{1/p^{\infty}})=B$. Since $g_{1}=p$ and colimits commute with tensor products, it follows that $A_{\infty}$ is quasi-isomorphic to the colimit, $B_{\infty}$, of dg $B$-algebras
\[
B_{n}:=K(g_2^{1/p^n},...,g_{d}^{1/p^n};\, B)
\]
where the maps  $B_{n}\to B_{n+1}$ are defined as for the $A_{n}$. It thus suffices to prove that the homology of $B_{\infty}$ is concentrated in degree $0$, where it is $B/(g_{2}^{1/p^\infty},\dots, g_{d}^{1/p^\infty})$. This is essentially the content of Lemma~\ref{lem:hochster}.  Indeed, as in the proof of Lemma~\ref{lem:hochster} one reduces to the case of a single element, $g$, in $B$. Consider  $F:=0\to (g^{1/p^{\infty}}) \xra{\subset} B \to 0$,  viewed as a dg $B$-algebra concentrated in degrees $0$ and $1$, and the morphism $B_{n}\to F$  of dg $B$-algebras
\[
\xymatrix{
0 \ar@{->}[r] & B \ar@{->}[d]_{1\mapsto g^{1/p^{n}}} \ar@{->}[r]^{g^{1/p^{n}}}
		&  B \ar@{=}[d]\ar@{->}[r] & 0 \\
0 \ar@{->}[r] &  (g^{1/p^{\infty}}) \ar@{->}[r]^{\subset } & B \ar@{->}[r] & 0
}
\]
It is clear that these morphisms are compatible with the morphisms $B_{n}\to B_{n+1}$, and so yield a morphism $B_{\infty}\to F$ of dg $B$-algebras. This an isomorphism: this is clear in degree $0$ zero, whilst in degree $1$ it was verified in the proof of Lemma~\ref{lem:hochster}.

This completes the proof of the claim. Observe that $A/(g_{1}^{1/p^\infty},\dots, g_{d}^{1/p^\infty}) \cong  (A/\fm A)_{\perf}$.

\medskip

In the remainder of the proof we use some basic facts about (strict graded-commutative) semifree dg $R$-algebras, referring to Avramov~\cite{Avramov:1998a} for details. Let $R[X]$ be a resolvent of $k$ viewed as an $R$-algebra; in particular,  the $R$-algebra  $R[X]$ is the strict graded-commutative polynomial ring on a graded set of indeterminates $X:=\{X_{i}\}_{i\geqslant 1}$. Since $R$ is noetherian, one can choose $X$ such that set $X_{i}$ is finite for each $i$; see \cite[Proposition~2.1.10]{Avramov:1998a}.

Claim~\ref{cl:acyclic} implies that the canonical surjection $A_{\infty}\to \HH 0{A_{\infty}}$  is a quasi-isomorphism of dg $A$-algebras, and hence also of dg $R$-algebras. By construction $\fm \HH 0{A_{\infty}}= 0$ so the induced morphism $R\to \HH 0{A_{\infty}}$ factors through the surjection $R\to k$. Since $R[X]$ is semifree, it then follows from \cite[Proposition~2.1.9]{Avramov:1998a} that there is a commutative square
\[
\xymatrix{
R[X] \ar[r]^{\varphi} \ar[d] & A_{\infty} \ar@{->>}[d]^{\simeq} \\
		  k \ar[r] & \HH 0{A_{\infty}}}
\]
of dg $R$-algebras. Recall that $A_{\infty}$ is constructed as a colimit of the $A_{n}$.

\begin{claim}
\label{cl:factorization}
For $n\gg 0$, the morphism $\varphi\colon R[X]\to A_{\infty}$ of dg $R$-algebras factors through $A_{n}$; that is to say, there is a morphism $\varphi_{n}\colon R[X]\to A_{n}$ of dg $R$-algebras such that its composition with $A_{n}\to A_{\infty}$ is $\varphi$.
\end{claim}
		  		
The crucial point is that since each complex $A_{n}$ is zero outside (homological) degrees $[0,d]$, so is their colimit $A_{\infty}$.  In particular, $\varphi(X_{i})=0$ for $i\ge d+1$ for degree reasons, and  $\varphi$ is completely determined by its values on the $X_{i}$ for $1\le i\le d$. Since each $X_{i}$ is finite, it clear that $\varphi$ lifts to a map of $R$-algebras $\varphi_{n}\colon R[X]\to A_{n}$ for some $n\ge 1$. Moreover, increasing $n$ if needed we can ensure that the commutator $[\partial,\varphi_{n}]$ vanishes on the $X_{i}$, that is to say, $\varphi_{n}$ is also a morphism of complexes, and hence a morphism of dg $R$-algebras.

\medskip

\begin{claim}
\label{cl:toriso}
For $n\gg 0$ there is an isomorphism of graded $A$-modules
\[
\Tor^{R}(A_{n},M)\cong \hh{A_{n}} \otimes_{k} \Tor^{R}(k,M)  \,.
\]
\end{claim}

By each $n$ as in Claim~\ref{cl:factorization}, the $R$-module structure on $A_{n}$ extends, via $\varphi_{n}$, to that of a dg module over $R[X]$, so one has
\[
A_n \lotimes_{R} M  \simeq A_{n} \lotimes_{R[X]}  (R[X] \lotimes_{R}  M)
\]
as complexes of $A$-modules. Since $\hh{R[X]}=k$ is a field, the K\"unneth map
\[
 \hh{A_n} \otimes_{\hh{R[X]}} \hh{R[X] \lotimes_{R} M} \longrightarrow \hh{A_{n} \lotimes_{R[X]}  (R[X] \lotimes_{R}  M)}
\]
is an isomorphism. Combining the preceding two isomorphisms yields
\[
\Tor^{R}(A_{n},M) = \hh{A_{n} \lotimes_{R} M}  \cong      \hh{A_{n}} \otimes_{k} \Tor^{R}(k,M)\,.
\]
This completes the proof of Claim~\ref{cl:toriso}.

\medskip

\emph{Proof of Claim~\ref{cl:key}}: Assume to the contrary that $\Tor_{s+d+1}^R(k,M) \neq 0$. Fix $n\gg 0$ so that Claim~\ref{cl:toriso} applies. Then the $A$-module $\Tor^{R}_{s+d+1}(A_{n},M)$ contains $\HH 0{A_{n}}$ as a direct summand.  Claim~\ref{cl:almostzero}  implies that $\HH 0{A_{n}}$ is almost zero. Since $\HH 0{A_n}\cong A/(g_{1}^{1/p^n},...,g_{d}^{1/p^n})$, by construction, this  fact translates to
\[
\sqrt{pA} = (g_{1}^{1/p^\infty})A  \subseteq (g_{1}^{1/p^n},...,g_{d}^{1/p^n})\,.
\]
Raising to a sufficiently large power and observing that $(g_1^{1/p^\infty})$ is idempotent, this gives
\[
(g_{1}^{1/p^\infty})A \subseteq (g_{1},...,g_{r})A = (f_{1},...,f_{r})A.
\]
This completes the proof of Claim~\ref{cl:key} and so also that of the theorem.
\end{proof}

The next result gives a way to check the valuative condition~\ref{ss:valuative}.

\begin{proposition}
\label{prp:AFFSpa}
Let $R$ be a noetherian $p$-torsionfree ring containing $p$ in its Jacobson radical.  Let $R\to A$ be a map, with $A$ perfectoid,  that is almost flat and $0$ is the only $R$-module $M$ for which  $M\otimes_{R}A=0$. Then the following statements hold.
\begin{enumerate}[\quad\rm(1)]
\item
$R\to A$ satisfies the valuative condition~\ref{ss:valuative}.
\item
$\sqrt{pA}\not\subseteq \fm A$ for any maximal ideal $\fm$ of $R$.
\item
$0$ is the only $R$-module $M$ for which $M\otimes_{R}A$ is almost zero.
\end{enumerate}
\end{proposition}

\begin{proof}
We proved that (2) follows from (1) as part of the proof of Theorem~\ref{thm:almostpadicKunz}. Here we establish (2) first, and then deduce (1) and (3) from it.

(2) This part does not use the hypothesis that $R\to A$ is almost flat. Let $\fp$ be a minimal prime of $\widehat{R_{\fm}}$, the $\fm$-adic completion of $R_{\fm}$ and set $S:=\widehat{R_{\fm}}/\fp$.
For $d=\dim S$ one gets that
\[
 \lch{\fm}d{S\otimes_RA}\cong \lch{\fm}d S \otimes_RA\neq 0\,;
\]
here we have used the hypothesis that $(-)\otimes_{R}A$ is faithful on modules. Therefore $S\otimes_R A$ is a solid $S$-algebra in the sense of Hochster; see~\cite[Corollary 2.4]{Hochster:1994a}.

Contrary to the desired result, suppose $\sqrt{pA}\subseteq \fm A$,  and consider the element $\varpi$ from Definition~\ref{def:perfectoid}.
Since $\varpi^{p}=pu$, for some unit $u$ in $A$, one has that $\sqrt{\varpi A}\subseteq \fm A$. Therefore, for each $e\ge 0$  the element $\varpi^{1/p^e}$ is in $\fm (S\otimes_{R}A)$ and hence  $\varpi\in \fm^{p^e}(S\otimes_{R}A)$; this implies that $p$ is the same ideal. By definition~\cite[(1.2)]{Hochster:1994a}, this implies that $p$ is contained in the solid closure of $\fm^{p^e}S$, and hence  in the integral closure of $\fm^{p^{e}}$, by \cite[Theorem 5.10]{Hochster:1994a}. This is a contradiction for $p$ is not zero in $S$; see \cite[Proposition~5.3.4]{Huneke/Swanson:2006a}.

\medskip

(3) Suppose $M\ne 0$, so there is an embedding $R/I\subseteq M$ for some nonunit ideal $I$ of $R$. Since $R\to A$ is almost flat, when $M\otimes_{R}A$ is almost zero, so is $(R/I)\otimes_{R}A$, that is to say, $(p^{1/p^\infty})A \subseteq IA$; this contradicts (2). Note that (2) is the special case $M=R/\fm$ of (3).

\medskip

(1) At this point we  can assume $R\to A$ is almost flat and that $0$ is the only $R$-module $M$ for which $M\otimes_{R}A$ almost zero. These hypotheses remain unchanged, and it suffices to verify the conclusion, when we replace $A$ by any almost isomorphic ring, and by doing so we can assume that $A$ is also $p$-torsion free. Then, for any integer $n\ge 1$ the map $R/p^{n}\to A/p^{n}$ is almost flat, and has the property that $M\otimes_{R/p^{n}}(A/p^{n})$ almost zero implies $M=0$. These observations will be used below.

 Fix a map $R \to V$ with $V$ a $p$-adically complete and $p$-torsionfree rank $1$ valuation ring. Set $B := A \otimes_R V$ and let $\widehat{B}$ denote the $p$-adic completion of $B$. It will be enough to show that $\widehat{B}[1/p] \neq 0$; then, for  any prime $\mathfrak{q}$ in $\widehat{B}$ not containing $p$, there exists a $p$-adic  rank $1$ valuation on the domain $\widehat{B}/\mathfrak{q}$,  for each maximal of this quotient contains $p$, and any such valuation extends $V$; see also~\cite[Proposition~3.6]{Huber:1993a}.

Assume towards contradiction that $\widehat{B}[1/p] = 0$. Then the Banach open mapping theorem shows that for some $m\ge 0$ one has $p^m \cdot \widehat{B} = 0$, that is to say, $\widehat{B} \simeq \widehat{B}/p^m$; see \cite{Bhatt:2018a}. Since $\widehat{B}/p^m \simeq B/p^m$ and the transition maps in the tower $\{B/p^n\}$ limiting to $\widehat{B}$ are surjective, it follows that $B/p^{m+1} \simeq B/p^m$ via the natural map. In other words, the surjective map $V/p^{m+1} \to V/p^m$ becomes an isomorphism after applying $- \otimes_{R/p^n} A/p^n$ for $n \geq m+1$.  It then follows that $p^m V/p^{m+1} V \otimes_{R/p^n} A/p^n$ is almost zero and hence that $p^m V/p^{m+1} V = 0$, which is absurd as $V$ is $p$-torsionfree and $p$-adically complete.
\end{proof}

Here is a more intuitive formulation of the $p$-adic Kunz theorem in the almost setting.

\begin{corollary}
\label{cor:almostpadicKunz}
Let $R$ be a noetherian $p$-torsionfree ring containing $p$ in its Jacobson radical.  If there exists a map $R \to A$ with $A$ perfectoid that is almost flat and zero is the only $R$-module $M$ for which $M\otimes_{R}A$ is zero, then $R$ is regular.
\end{corollary}

\begin{proof}
This is a direct consequence of Proposition~\ref{prp:AFFSpa}(1) and Theorem~\ref{thm:almostpadicKunz}.
\end{proof}


\begin{bibdiv}
\begin{biblist}


\bib{Aberbach/Li:2008a}{article}{
   author={Aberbach, Ian M.},
   author={Li, Jinjia},
   title={Asymptotic vanishing conditions which force regularity in local
   rings of prime characteristic},
   journal={Math. Res. Lett.},
   volume={15},
   date={2008},
   number={4},
   pages={815--820},
   issn={1073-2780},
   review={\MR{2424915}},
}

\bib{Aberbach/Hochster:1997a}{article}{
   author={Aberbach, Ian M.},
   author={Hochster, Melvin},
   title={Finite Tor dimension and failure of coherence in absolute integral
   closures},
   journal={J. Pure Appl. Algebra},
   volume={122},
   date={1997},
   number={3},
   pages={171--184},
   issn={0022-4049},
   review={\MR{1481086}},
}
%


\bib{Andre:2018a}{article}{
    author={Andr\'e, Yves},
    title={Perfectoid spaces and the homological conjectures},
    status={preprint 2018},
    eprint ={arxiv:1801.10006v1},
}

\bib{Andre:2016a}{article}{
    author={Andr\'e, Yves},
   title={Le lemme d'Abhyankar perfectoide},
   language={French},
   journal={Publ. Math. Inst. Hautes \'Etudes Sci.},
   volume={127},
   date={2018},
   pages={1--70},
   issn={0073-8301},
   review={\MR{3814650}},
   doi={10.1007/s10240-017-0096-x},
}

\bib{Avramov:1998a}{article}{
   author={Avramov, Luchezar L.},
   title={Infinite free resolutions},
   conference={
      title={Six lectures on commutative algebra},
      address={Bellaterra},
      date={1996},
   },
   book={
      series={Progr. Math.},
      volume={166},
      publisher={Birkh\"auser, Basel},
   },
   date={1998},
   pages={1--118},
   review={\MR{1648664}},
}
\bib{Avramov/Iyengar/Miller:2006a}{article}{
   author={Avramov, Luchezar L.},
   author={Iyengar, Srikanth},
   author={Miller, Claudia},
   title={Homology over local homomorphisms},
   journal={Amer. J. Math.},
   volume={128},
   date={2006},
   number={1},
   pages={23--90},
   issn={0002-9327},
   review={\MR{2197067}},
}

\bib{Bass/Murthy:1967a}{article}{
   author={Bass, Hyman},
   author={Murthy, M. Pavaman},
   title={Grothendieck groups and Picard groups of abelian group rings},
   journal={Ann. of Math. (2)},
   volume={86},
   date={1967},
   pages={16--73},
   issn={0003-486X},
   review={\MR{0219592}},
}

\bib{Bhatt:2018a}{article}{
    author={Bhatt, Bhargav},
    title={Torsion  completions are bounded},
    journal={J. Pure Appl. Algebra},
    status={to appear},
}


\bib{Bhatt:2016a}{article}{
    author={Bhatt, Bhargav},
   title={On the direct summand conjecture and its derived variant},
   journal={Invent. Math.},
   volume={212},
   date={2018},
   number={2},
   pages={297--317},
   issn={0020-9910},
   review={\MR{3787829}},
   doi={10.1007/s00222-017-0768-7},
}

\bib{Bhatt/Morrow/Scholze:2016a}{article}{
 author={Bhatt, Bhargav},
 author={Morrow, Matthew},
 author={Scholze, Peter},
 title={Integral $p$-adic Hodge theory},
 status={preprint 2016},
 eprint={arxiv:1602.03148v1},
}


\bib{Bhatt/Morrow/Scholze:2018}{article}{
 author={Bhatt, Bhargav},
 author={Morrow, Matthew},
 author={Scholze, Peter},
 title={Topological Hochschild homology and integral $p$-adic Hodge theory},
 status={preprint 2018},
 eprint={arxiv:1802.03261},
}



\bib{Bhatt/Scholze:2017a}{article}{
   author={Bhatt, Bhargav},
   author={Scholze, Peter},
   title={Projectivity of the Witt vector affine Grassmannian},
   journal={Invent. Math.},
   volume={209},
   date={2017},
   number={2},
   pages={329--423},
   issn={0020-9910},
   review={\MR{3674218}},
}


\bib{Bourbaki:CA9}{book}{
   author={Bourbaki, N.},
   title={\'El\'ements de math\'ematique. Alg\`ebre commutative. Chapitres 8 et 9},
   language={French},
   note={Reprint of the 1983 original},
   publisher={Springer, Berlin},
   date={2006},
   pages={ii+200},
   isbn={978-3-540-33942-7},
   isbn={3-540-33942-6},
   review={\MR{2284892}}, }

\bib{Bruns/Herzog:1998a}{book}{
   author={Bruns, Winfried},
   author={Herzog, J{\"u}rgen},
   title={Cohen-Macaulay rings},
   series={Cambridge Studies in Advanced Mathematics},
   volume={39},
   edition={2},
   publisher={Cambridge University Press, Cambridge},
   date={1998},
   pages={xii+403},
   isbn={0-521-41068-1},
   review={\MR{1251956}},
}

\bib{Christensen/Iyengar/Marley:2018a}{article}{
    author={Christensen, Lars Winther},
    author={Iyengar, Srikanth B.},
    author={Marley, Thomas},
    title={Rigidity of Ext and Tor with coefficients in residue fields of a commutative noetherian ring},
    journal={Proc. Edinb. Math. Soc. (2)},
    status={to appear},
    eprint ={ arxiv:1611.03280},
   }

\bib{Conrad:2014a}{webpage}{
author={Conrad, Brian},
url= {http://virtualmath1.stanford.edu/~conrad/Perfseminar/Notes/L10.pdf},
}


\bib{Greenlees/May:1992a}{article}{
author={Greenlees, J. P. C.},
   author={May, J. P.},
   title={Derived functors of $I$-adic completion and local homology},
   journal={J. Algebra},
   volume={149},
   date={1992},
   number={2},
   pages={438--453},
   issn={0021-8693},
   review={\MR{1172439}},
}

\bib{Cuong/Cuong:2017a}{article}{
   author={Cuong, Nguyen Tu},
   author={Cuong, Doan Trung},
   title={Local cohomology annihilators and Macaulayfication},
   journal={Acta Math. Vietnam.},
   volume={42},
   date={2017},
   number={1},
   pages={37--60},
   issn={0251-4184},
   review={\MR{3595445}},
}

\bib{Gabber/Ramero:2017}{book}{
  author={Gabber, Ofer},
  author={Ramero, Lorenzo},
  title={Foundations of almost ring theory},
  date={2017},
  status={preprint},
  eprint={arxiv:0409584v12},
  }

\bib{Fontaine:2013}{article}{
 author={Fontaine, Jean Marc},
 title={Perfect\"{o}ides, presque puret\'e et monodromie-poids (d'apr\'es Peter Scholze)},
 journal={Ast\'erisque},
 volume={352},
 year={2013},
 pages={1043-1058},
  }

\bib{Heitmann:2005a}{article}{
   author={Heitmann, Raymond C.},
   title={Extended plus closure and colon-capturing},
   journal={J. Algebra},
   volume={293},
   date={2005},
   number={2},
   pages={407--426},
   issn={0021-8693},
   review={\MR{2172347}},
}

\bib{Heitmann:2002a}{article}{
   author={Heitmann, Raymond C.},
   title={The direct summand conjecture in dimension three},
   journal={Ann. of Math. (2)},
   volume={156},
   date={2002},
   number={2},
   pages={695--712},
   issn={0003-486X},
   review={\MR{1933722}},
}

\bib{Herzog:1974a}{article}{
   author={Herzog, J\"urgen},
   title={Ringe der Charakteristik $p$ und Frobeniusfunktoren},
   language={German},
   journal={Math. Z.},
   volume={140},
   date={1974},
   pages={67--78},
   issn={0025-5874},
   review={\MR{0352081}},
}

\bib{Hochster:1994a}{article}{
   author={Hochster, Melvin},
   title={Solid closure},
   conference={
      title={Commutative algebra: syzygies, multiplicities, and birational
      algebra},
      address={South Hadley, MA},
      date={1992},
   },
   book={
      series={Contemp. Math.},
      volume={159},
      publisher={Amer. Math. Soc., Providence, RI},
   },
   date={1994},
   pages={103--172},
   review={\MR{1266182}},
   doi={10.1090/conm/159/01508},
}


\bib{Hochster/Huneke:1992a}{article}{
   author={Hochster, Melvin},
   author={Huneke, Craig},
   title={Infinite integral extensions and big Cohen-Macaulay algebras},
   journal={Ann. of Math. (2)},
   volume={135},
   date={1992},
   number={1},
   pages={53--89},
   issn={0003-486X},
   review={\MR{1147957}},
}

\bib{Huber:1993a}{article}{
   author={Huber, R.},
   title={Continuous valuations},
   journal={Math. Z.},
   volume={212},
   date={1993},
   number={3},
   pages={455--477},
   issn={0025-5874},
   review={\MR{1207303}},
   doi={10.1007/BF02571668},
}

\bib{Huneke:1996a}{article}{
   author={Huneke, Craig},
   title={Tight closure, parameter ideals, and geometry},
   conference={
      title={Six lectures on commutative algebra},
      address={Bellaterra},
      date={1996},
   },
   book={
      series={Progr. Math.},
      volume={166},
      publisher={Birkh\"auser, Basel},
   },
   date={1998},
   pages={187--239},
   review={\MR{1648666}},
}

\bib{Huneke/Lyubeznik:2007a}{article}{
   author={Huneke, Craig},
   author={Lyubeznik, Gennady},
   title={Absolute integral closure in positive characteristic},
   journal={Adv. Math.},
   volume={210},
   date={2007},
   number={2},
   pages={498--504},
   issn={0001-8708},
   review={\MR{2303230}},
}

\bib{Huneke/Swanson:2006a}{book}{
   author={Huneke, Craig},
   author={Swanson, Irena},
   title={Integral closure of ideals, rings, and modules},
   series={London Mathematical Society Lecture Note Series},
   volume={336},
   publisher={Cambridge University Press, Cambridge},
   date={2006},
   pages={xiv+431},
   isbn={978-0-521-68860-4},
   isbn={0-521-68860-4},
   review={\MR{2266432}},
}

\bib{Kawasaki:2002a}{article}{
   author={Kawasaki, Takesi},
   title={On arithmetic Macaulayfication of Noetherian rings},
   journal={Trans. Amer. Math. Soc.},
   volume={354},
   date={2002},
   number={1},
   pages={123--149},
   issn={0002-9947},
   review={\MR{1859029}},
}

\bib{Kedlaya/Liu:2015}{article}{
  author={Kedlaya, Kiran},
  author={Liu, Ruochuan},
  title={Relative $p$-adic Hodge theory},
  journal={Ast\'erisque},
  volume={371},
  year={2015},
  pages={111-165},
  }


\bib{Kunz:1969a}{article}{
   author={Kunz, Ernst},
   title={Characterizations of regular local rings for characteristic $p$},
   journal={Amer. J. Math.},
   volume={91},
   date={1969},
   pages={772--784},
   issn={0002-9327},
   review={\MR{0252389}},
}

\bib{Lau:2016a}{article}{
 author={Lau, Eike},
 title={Dieudonn\'e theory over semiperfect rings and perfectoid rings},
 status={preprint 2016},
 eprint={arXiv:1603.07831},
 }

\bib{Morrow:2016}{article}{
  author={Morrow, Matthew},
  title={Notes on the $A_{\inf}$-cohomology of integral $p$-adic Hodge theory},
  status={preprint 2016},
  eprint={arxiv:1608.00922},
  }



\bib{Nagata:1975a}{book}{
   author={Nagata, Masayoshi},
   title={Local rings},
   note={Corrected reprint},
   publisher={Robert E. Krieger Publishing Co., Huntington, N.Y.},
   date={1975},
   pages={xiii+234},
   isbn={0-88275-228-6},
   review={\MR{0460307}},
}
\bib{Schenzel:2003a}{article}{
   author={Schenzel, Peter},
   title={Proregular sequences, local cohomology, and completion},
   journal={Math. Scand.},
   volume={92},
   date={2003},
   number={2},
   pages={161--180},
   issn={0025-5521},
   review={\MR{1973941}},
}

\bib{Scholze:2017}{article}{
 author={Scholze, Peter},
 title={\'Etale cohomology of diamonds},
 status={preprint 2017},
 eprint={arXiv:1709.07343 },
 }


\bib{Shimomoto:2014}{article}{
  author={Shimomoto, Kazuma},
   title={An application of the almost purity theorem to the homological
   conjectures},
   journal={J. Pure Appl. Algebra},
   volume={220},
   date={2016},
   number={2},
   pages={621--632},
   issn={0022-4049},
   review={\MR{3399381}}, }

\end{biblist}
\end{bibdiv}
\end{document}